\documentclass[preprint,12pt]{elsarticle}
\usepackage{hyperref}
\usepackage{amsmath}    % need for subequations
\usepackage{graphicx}   % need for figures
\usepackage{graphics}
\usepackage{amsfonts}
\usepackage{array}      % use for table
\usepackage{verbatim}   % useful for program listings
\usepackage{color}      % use if color is used in text
\usepackage{subfigure}  % use for side-by-side figures
\usepackage{mathtools}
\usepackage{amsthm}
\usepackage{amsxtra}
\usepackage{amstext}
\usepackage{amssymb}
\usepackage{enumerate}
\usepackage[T1]{fontenc}
\usepackage[utf8]{inputenc}
\usepackage[font=small,labelfont=bf,tableposition=top]{caption}
\usepackage{booktabs,multirow}
\begin{document}

\begin{frontmatter}
\title
{Approximate solutions of a time-fractional diffusion equation with a source term using the variational iteration method}

\author
{Iftikhar Ali, Bilal Chanane, Nadeem A. Malik$^1$}

\address{Department of Mathematics and Statistics,
King Fahd University of Petroleum and Minerals,
P.O. Box 5046, Dhahran 31261, Saudi Arabia}

\fntext[label2]
{Corresponding author e-mail: namalik@kfupm.edu.sa (or nadeem\_malik@cantab.net)
\hfill}

\begin{abstract}
We consider a time fractional differential equation of order $\alpha$, $0<\alpha<1$, 
$$
\frac{\partial c(x,t)}{\partial t}={}^C_0\mathcal{D}_t^{\alpha}[(Ac)(x,t)]+q(x,t) ,\quad x > 0, t > 0, \quad c(x,0)=f(x).
$$
where ${}^C_0\mathcal{D}_t^{\alpha}$ is the Caputo fractional derivative of order $\alpha$,
$A$ is a linear differential operator, $q(x,t)$ is a source term,
and $f(x)$ is the inital condition. Approximate (truncated)
series solutions are obtained by means of the Variational Iteration Method
(VIM). We find the series solutions for different cases of the source term,
in a form that is readily implementable on the computer where
symbolic computation platform is available. The error in  truncated
solution $c_n$ diminishes exponentially fast for a given $\alpha$
as the number of terms in the series increases. VIM has several advantages
over other methods that produce solutions in the series form. The truncated
VIM solutions often converge rapidly requiring only a few terms for fast and
accurate approximations.
\end{abstract}

%\maketitle \numberwithin{equation}{section}
\newtheorem{theorem}{Theorem}[section]
\newtheorem{lemma}[theorem]{Lemma}
\newtheorem{remark}[theorem]{Remark}
\newtheorem{definition}[theorem]{Definition}
\newtheorem{corollary}[theorem]{Corollary}
\allowdisplaybreaks
\end{frontmatter}

\section{Introduction}\label{intro}

Recently many researchers have formulated mathematical models for
a wide range of different physical phenomena using fractional calculus, from
crowded systems to transport through porous media.
For example, Metzler and Klafter \cite{metzler2000random},
derived the fractional partial differential equations that describes
anomalous diffusion through porous media; Mainardi
\cite{mainardi2012historical} used fractional models to describe waves
propagating through viscoelastic materials; Hilfer \cite{hilfer2000fractional} provided
many applications of fractional calculus in physics. Similarly,
there are applications of fractional calculus in biology Magin
\cite{magin2010fractional}, in medical sciences Magin and Ovadia \cite{magin2008modeling},
in ecological modeling Agrawal et al. \cite{agrawal2012synchronization},
in finance  Scalas et al. \cite{scalas2000fractional}. Ross \cite{ross1975fractional} has mentioned  a number of areas where fractional
 calculus is useful in order to analyze a system; mathematical physics,
 spherical (radial) probability modes generations, hyperstereology, modeling of holograph linearities.
Das  \cite{das2011functional} has discussed the applications of fractional calculus
 in engineering problems, especially, evolutionary design of combinational circuits, electrical  skin phenomena, field programmable gate arrays.

Many different notions of fractional derivatives are given in the literature,
see Kilbas et al. \cite{kilbas2006fracderi},  but the most used definitions are Riemann-Liouville
fractional derivative and the Caputo fractional derivative, defined in the
Section \ref{prelim}. Hilfer \cite{hilfer2000fractional} proposed  the idea of generalized Riemann-Liouville
derivative which is essentially an interpolation between Riemann fractional derivative
and Caputo fractional derivative, and sometimes in the literature it is referred
as the Hilfer fractional derivative . See Hilfer \cite{hilfer2013applications} for a recent
account on Hilfer fractional derivatives.

Furthermore, with the advent of new fractional methods, there is a need to develop
efficient, fast and stable numerical algorithms for the integration
of fractional differential equations. Therefore, in parallel researchers are developing
new semi-analytical and numerical methods to
find the solutions of  proposed mathematical models that are
based on fractional calculus. For instance, He
\cite{he1998approximate} proposed a new analytic method called
Variational Iteration Method (VIM) to find the approximate
solution of the fractional nonlinear differential equations. Odibat and Momani
\cite{odibat2006application} used VIM to obtain the solution of
different time-fractional differential equation and made a
comparison with other methods such as
Adomian decomposition method and homotopy perturbations methods,
 see Momani and Odibat \cite{momani2007comparison}.

In the present work, we study a time-fractional diffusion equation with a source term,
\begin{gather}\label{intro1}
\frac{\partial}{\partial
t}c(x,t)={}^C_{0}\mathcal{D}_t^{\alpha}[(Ac)(x,t)]+q(x,t) ,\quad x
> 0, t > 0,\\
c(x,0)=f(x),
\end{gather}
where $^C_0\mathcal{D}_t^{\alpha}$ denotes the  Caputo %Riemann-Liouville
fractional derivative (defined below in Eq. (5)),
$A$ represents a linear operator in the spatial variable $x$, $q(x,t)$
represents the source or sink term and $f(x)$ represents the initial
condition.  The unknown function $c(x,t)$, is also called a
propagator, Metzler and Klafter \cite{metzler2004restaurant}, Luchko and Punzi \cite{luchko2011modeling},
and it can be interpreted as a diffusing scalar (e.g. temperature, passive particle)
or as the probability density function of locating a particle at the
position $x$ at the time $t$.

The main objectives of the present study are, firstly
to find the  approximate analytic solution of equation
\eqref{intro1} for some specific cases of the linear operator $A$ and
source term $q(x,t)$ using VIM and secondly to express the series solutions
in a form that is easy to implement on computer.
Thirdly, we present a case study of sinusoidal uploading whose exact solution is
 known which can be used to compare the accuracy of truncated series solutions obtained by VIM.

 We have organized this article as follows: in Section (\ref{prelim}),
 we provide some basic definitions and results from
 fractional calculus, in Section (\ref{vim}), we describe
the variational iteration method to obtain  the solution of the
problem \eqref{intro1} subject to initial condition, in Section
(\ref{case}), we provide a case study of  a time-fractional differential equation
with sinusoidal uploading, we find the approximate solutions by variational iteration method and
then compare the results with exact solution. We plot
the graphs of the VIM solutions along with exact solution, moreover,
we provide the error plots. In the last
Section (\ref{conc}), we state our conclusions of the study.

\section{Preliminaries}\label{prelim}
In this section, we state few definitions and results from
fractional calculus. A detailed account on fractional derivatives and integrals
can be found in Kilbas et al. \cite{kilbas2006fracderi}. The generalized
derivatives with their Laplace transforms are discussed in Sandev et al.
\cite{sandev2011fractional}. \\ \\

\noindent\textbf{\emph{Riemann-Liouville Fractional Integral }}
 of order $\alpha$ for an absolutely integrable function $f(t)$ is
defined by
\begin{equation}\label{pre1}
\left(_0I_t^{\alpha}f\right)(t):=\frac{1}{\Gamma(\alpha)}\int_0^t(t-\tau)^{\alpha-1}f(\tau)d\tau,\quad
t > 0, \alpha\in\mathbb{R}^+
\end{equation}
where $\mathbb{R}^+$ is the set of positive real numbers.\\

\noindent\textbf{\emph{Riemann-Liouville Fractional Derivative }} of
order $\alpha>0$ for an absolutely integrable function $f(t)$ is
defined by
\begin{align}\label{pre2}
\nonumber \left(_0D_t^{\alpha}f\right)(t):&=D_t^m \circ
_0I_t^{m-\alpha}f(t),\\
&=\left\{
    \begin{array}{ll}
        \frac{1}{\Gamma(m-\alpha)}\frac{d^m}{dt^m}\int_{0}^{t}\frac{f(\tau)}{(t-\tau)^{\alpha+1-m}}d\tau  & \mbox{if } m-1<\alpha<m  \\
        \frac{d^m}{dt^m}f(t) & \mbox{if } \alpha=m
    \end{array}
\right.
\end{align}\\

\noindent\textbf{\emph{Caputo Fractional Derivative }}
of order $\alpha>0$ for a function $f(t)$, whose $m$th order
derivative is absolutely integrable, is defined by
\begin{align}\label{pre3}
\nonumber \left(_0^CD_t^{\alpha}f\right)(t):&=
_0I_t^{m-\alpha}\circ D_t^mf(t),\\
&=\left\{
    \begin{array}{ll}
        \frac{1}{\Gamma(m-\alpha)}\int_{0}^{t}\frac{f^{(m)}(\tau)}{(t-\tau)^{\alpha+1-m}}d\tau  & \mbox{if } m-1<\alpha<m  \\
        \frac{d^m}{dt^m}f(t) & \mbox{if } \alpha=m
    \end{array}
\right.
\end{align}
In general, Riemann-Liouville and Caputo fractional derivatives are not equal, i.e.,
\[
\left(_0D_t^{\alpha}f\right)(t):=D_t^m \circ
_0I_t^{m-\alpha}f(t)\neq _0I_t^{m-\alpha}\circ
D_t^mf(t)=:\left(_0^CD_t^{\alpha}f\right)(t)
\]
unless $f(t)$ along with its $m-1$ derivatives vanish at $t=0^+$.\\

\noindent\textbf{\emph{Hilfer Fractional Derivative }}
 of order $\alpha$, $0<\alpha<1$ and type $\beta$, $0\leq\beta\leq1$ for an absolutely
integrable function $f(t)$ with respect to $t$ is defined by,
\cite{hilfer2013applications},
\begin{equation}\label{pre4}
\left(D_t^{\alpha,\beta}f\right)(t)=\left(_0I_t^{\beta(1-\alpha)}\frac{d}{dt}I^{(1-\beta)(1-\alpha)}f\right)(t).
\end{equation}
Note that Hilfer fractional derivative interpolates between
Riemann-Liouville fractional derivative and Caputo fractional
derivative, because if $\beta=0$ then Hilfer fractional derivative
corresponds to Riemann-Liouville fractional derrivative and if
$\beta=1$ then Hilfer fractional derivative corresponds to Caputo
fractional derivative.\\

\noindent\textbf{\emph{Riemann-Liouville derivative of a constant A:}} \\ \\
$ {}_0D_t^{\alpha} A = \displaystyle{At^{-\alpha}\over \Gamma(1-\alpha)}$\\ \\

\noindent\textbf{\emph{For the Caputo derivative we have:}}\\ \\
 ${}^C_0D_t^{\alpha} \ A = 0$, where $A$ is a constant. \\

\noindent ${}^C_0D_t^{\alpha} \ t^\beta = \displaystyle{\Gamma(\beta+1) 
\over \Gamma(\beta-\alpha+1)} x^{\beta-\alpha}$ for $n-1<\alpha <n$, $\beta > n-1$
\\ \\

\noindent\textbf{\emph{Mittag-Leffler Function}} is the generalization of
exponential function\\ $\displaystyle{e^z=\sum_{k=0}^{\infty}\frac{z^k}{k!}}$.\\

\noindent\textbf{\emph{1-parameter Mittag-Leffler Function }}
\begin{align}\label{pre5}
E_{\nu}(z)=\sum_{k=0}^{\infty}\frac{z^k}{\Gamma(\nu k+1)},
\quad \nu >0.
\end{align}
\vskip 0.2cm \noindent\textbf{\emph{2-parameter Mittag-Leffler
Function }}
\begin{align}\label{pre6}
E_{\nu,\mu}(z)=\sum_{k=0}^{\infty}\frac{z^k}{\Gamma(\nu
k+\mu)}, \quad \nu>0, \mu>0.
\end{align}

%%%%%%%%%%%%%%%%%%%%%%%%%%%%%%%%%%%%%%%%%%%%%%%%%%%%%%

\section{Variational Iteration Method}\label{vim}

Variational iteration method is an analytic method for finding the
solutions of differential equations. It poses the given differential
equation in an iterative integral form with the initial guess. It generates
a sequence of approximate solutions which eventually converge to the exact
solution provided the solution exists. The $n$th order truncated series
can be used to estimate the solution of the given differential equation.
The method can be used to find the solutions of linear or nonlinear,
conventional or fractional, ordinary or partial differential equations.

 In this section, we describe the variational iteration method,
and provide an outline for its implementation. He \cite{he1998approximate}
proposed VIM to obtain the solutions of fractional differential
equations describing the seepage flow in porous media. Later, He \cite{he1999variational}
extended the method to nonlinear differential equations and obtained
the analytic solutions of some nonlinear differential equations.
The method provides the solution in the form of a
rapidly convergent successive approximations. For problems where a
closed form of the exact solution is not achievable, the $n$th  approximation
can be used to estimate the exact solution.

 Variational iteration method has certain advantages over the other proposed
analytic methods such as Adomian decomposition method (ADM) and  homotopy
perturbation method (HPM). In the case of ADM, a lot of work has to be done in order
to compute the Adomian polynomials for nonlinear terms,  see Wazwaz \cite{wazwaz2009partial},
and in the case of homotopy perturbation method (HPM), the method requires a huge amount of
calculations when the degree of nonlinearity increases, Momani and Odibat \cite{momani2007comparison}.
On the other hand, no specific requirements are needed, for nonlinear operators,
in order to use VIM, for instance, HPM requires an introduction of small parameter, or
the assumption of linearity in other nonlinear methods.

 Variational iteration method has been widely acknowledged and it has been extensively used
in all branches of science and engineering to find the solutions
of differential equations. For instance,  Noor and Mohyud-Din \cite{noor2010variational} applied VIM to solve the
twelfth order boundary value problems using He's polynomials. Shirazian and Effati \cite{shirazian2012solving}
solved a class of nonlinear optimal control problems by using VIM. Sakar et al. \cite{sakar2012variational}
obtained the approximate analytical solutions of the nonlinear Fornberg-Whitham
equation with fractional time derivative.  Chen and Wang \cite{chen2010variational} employed VIM
for solving a neutral functional-differential equation with proportional delays.
Elsaid \cite{elsaid2010variational} used VIM for solving Riesz fractional partial
differential equations. Noor and Mohyud-Din \cite{noor2009variational} used VIM for solving problems
related to unsteady  flow of gas through a porous medium using He's polynomials
and Pade approximants. VIM also proves to be effective for the heat
and the wave equations, see Molliq  et al.\cite{molliq2009variational}.

 Next, we describe the procedure how to use VIM to the problem \eqref{intro1}.
Consider the time-fractional partial differential equation,
\begin{equation}\label{vim1}
\frac{\partial}{\partial t}c(x,t)= {}_0^C\mathcal{D}_t^{\alpha}[(Ac)(x,t)]+q(x,t),
\end{equation}
where $^C_0\mathcal{D}_t^{\alpha}$ represents the Caputo %Riemann-Liouville
fractional derivative with respect to the time variable $t$, and $A$
represents a differential operator with respect to the space variable $x$.
\vskip 0.2cm \par The variational iteration method presents a
correctional functional in $t$ for Eq. \eqref{vim1}
in the form, with $c_n$ assumed known,
\begin{equation}\label{vim2}
c_{n+1}(x,t)=c_{n}(x,t)+\int_0^t\lambda(\xi)\left(\frac{\partial{c_n(x,\xi)}}{\partial{\xi}}-
{}^C_0\mathcal{D}_{\xi}^{\alpha}[(A\widetilde{c_n})(x,\xi)]-q(x,\xi)\right)d\xi,
\end{equation}
where $\lambda(\xi)$ is a general Lagrange multiplier which can be
identified optimally by variational theory and $\tilde{c}_n$ is a
restricted value that means it behaves like a constant, hence
$\delta\widetilde{c}_n=0$, where $\delta$  is the variational
derivative.

VIM is implemented in two basic steps;
\begin{enumerate}
\item The determination of the Lagrange multiplier $\lambda(\xi)$
that will be identified optimally through variational theory.
\item With $\lambda(\xi)$ determined, we substitute the result
into Eq. \eqref{vim2} where the restriction should be
omitted.
\end{enumerate}
\vskip 0.2cm \par Taking the $\delta-$variation of Eq. \eqref{vim2}
with respect to $c_n$, we obtain
\begin{equation}\label{vim3}
\delta c_{n+1}(x,t)=\delta
c_{n}(x,t)+\delta\int_0^t\lambda(\xi)\left(\frac{\partial{c_n(x,\xi)}}{\partial{\xi}}-
{}^C_0\mathcal{D}_{\xi}^{\alpha}[A(\widetilde{c_n}(x,\xi))]-q(x,\xi)\right)d\xi.
\end{equation}
\vskip 0.2cm \par Since $\delta \widetilde{c_n}=0$ and $\delta q=0$, we have
\begin{equation}\label{vim4}
\delta c_{n+1}(x,t)=\delta
c_{n}(x,t)+\delta\int_0^t\lambda(\xi)\left(\frac{\partial{c_n(x,\xi)}}{\partial{\xi}}\right)d\xi.
\end{equation}
To determine the Lagrange multiplier
$\lambda(\xi)$ we integrate by parts the integral in the Eq.
\eqref{vim4}, and noting that variational derivative of a constant
is zero, that is, $\delta\widetilde{c}_n=0$. Hence the Eq. \eqref{vim4}
yields
\begin{align}\label{vim5}
\nonumber \delta c_{n+1}(x,t)&=\delta
c_{n}(x,t)+\delta\left(\lambda(\xi)c_n(x,\xi)|_{\xi=t}-\int_0^t\frac{\partial}{\partial\xi}\lambda(\xi)\delta
c_n(x,\xi)\right)d\xi\\
&=\delta
c_{n}(x,t)(1+\lambda(\xi)|_{\xi=t})-\int_0^t\frac{\partial}{\partial\xi}\lambda(\xi)\delta
c_n(x,\xi)d\xi
\end{align}
\vskip 0.2cm \par The extreme values of $c_{n+1}$ requires that
$\delta c_{n+1}=0$. This means that left hand side of equation
\eqref{vim5} is zero, and as a result the right hand side should
be zero as well, that is,
\begin{equation}\label{vim6}
\delta
c_{n}(x,t)(1+\lambda(\xi)|_{\xi=t})-\int_0^t\frac{\partial}{\partial\xi}\lambda(\xi)\delta
c_n(x,\xi)d\xi=0
\end{equation}
\vskip 0.2cm \par This yields the stationary conditions
\begin{gather}
1+\lambda(\xi)|_{\xi=t}=0\\
\text{and }\lambda'(\xi)=0\\
\text{which implies } \lambda=-1.
\end{gather}
\vskip 0.2cm \par Hence Eq. \eqref{vim2} becomes
\begin{equation}\label{vim7}
c_{n+1}(x,t)=c_{n}(x,t)-\int_0^t\left(\frac{\partial{c_n(x,\xi)}}{\partial{\xi}}-
{}^C_0\mathcal{D}_{\xi}^{\alpha}[A(c_n(x,\xi))]-q(x,\xi)\right)d\xi,
\end{equation}
where the restriction is removed on $c_n$. Equation \eqref{vim7}
can further be simplified into the following form:
\begin{equation}\label{vim8}
c_{n+1}(x,t)=c_{n}(x,0)+\int_0^t\left(
{}^C_0\mathcal{D}_{\xi}^{\alpha}[(Ac_n)(x,\xi)]\right)d\xi + \int_0^tq(x,\xi)d\xi,
\end{equation}
for $n\geq 0$. We can use Eq. \eqref{vim8} to obtain the
successive approximations of the solution of the problem
\eqref{vim1}. The zeroth approximation $c_0(x,t)$ can be chosen
from the initial condition.

  Introducing the notation $J_t(\cdot)=\int_0^t(\cdot)d\xi$, Eq. \eqref{vim8}
can be rewritten as
\begin{equation}\label{vim9}
c_{n+1}(x,t)=c_{n}(x,0) + J_t\left({}^C_0\mathcal{D}_{t}^{\alpha}[(Ac_n)(x,t)]\right)+J_t\left(q(x, t)\right).
\end{equation}

Setting $e_n(x,t)=c_n(x,t) - c_{n-1}(x,t)$, for $n\geq 1$ and $c_0(x,t)=f(x)$,
we have
\begin{equation}\label{vim10}
e_{n+1}(x,t)=A^n \left(J_t\ {}^C_0\mathcal{D}_{t}^{\alpha}\right)^n e_1(x,t),
\end{equation}
for $n\geq 1$.

 Let $q(x,t)$ be analytic in $t$ about $t=0$, we have
 \begin{equation}\label{vim11}
q(x,t) = \sum_{k\geq 0}q_k(x) \frac{t^k}{k!}.
 \end{equation}

 Notice that the Riemann-Liouville derivative of $t^{\lambda}$  is given by,
\begin{equation}\label{vim12}
{}^C_0\mathcal{D}_t^{\alpha}(t^{\lambda})=\frac{\Gamma(1+\lambda)}{\Gamma(1+\lambda-\alpha)}t^{\lambda-\alpha}.
\end{equation}
Integrating above expression from $0$ to $t$, we obtain
\begin{equation}\label{vim13}
J_t \left({}^C_0\mathcal{D}_t^{\alpha} t^{\lambda} \right) = \int_0^t {}^C_0\mathcal{D}_{\xi}^{\alpha}({\xi}^{\lambda})d\xi=\frac{\Gamma(1+\lambda)}{\Gamma((1-\alpha)+\lambda+1)}t^{\lambda+(1-\alpha)}.
\end{equation}

 We claim the following:

\begin{lemma}\label{lma1}
\begin{equation}\label{vim14}
\left( J_t\ {}^C_0\mathcal{D}_t^{\alpha}\right)^n t^{\lambda}=\frac{\Gamma(1+\lambda)}{\Gamma(1+\lambda + n(1-\alpha))}t^{\lambda+n(1-\alpha)}
\end{equation}
for $n\geq 1$.
\end{lemma}

\begin{proof}
We prove it by induction on $n$. The relation \eqref{vim14} is true for $n=1$ by Eq. \eqref{vim13}.
Assume it is true for $n-1$, that is,
\begin{equation}\label{vim15}
\left(J_t\ {}^C_0\mathcal{D}_t^{\alpha}\right)^{n-1} t^{\lambda}=\frac{\Gamma(1+\lambda)}{\Gamma(1+\lambda + (n-1)(1-\alpha))}t^{\lambda + (n-1)(1-\alpha)}.
\end{equation}
We shall prove it true for $n$; we have by using Eq. \eqref{vim15}
\begin{align}\label{vim16}
\nonumber \left(J_t\  {}^C_0\mathcal{D}_t^{\alpha}\right)^n t^{\lambda} & =J_t\ {}^C_0\mathcal{D}_t^{\alpha}\left(\frac{\Gamma(1+\lambda)}{\Gamma(1+\lambda + (n-1)(1-\alpha))}t^{\lambda + (n-1)(1-\alpha)}\right)\\
\nonumber & = \frac{\Gamma(1+\lambda)}{\Gamma(1+\lambda + (n-1)(1-\alpha))}J_t\ {}^C_0\mathcal{D}_t^{\alpha}\left(t^{\lambda + (n-1)(1-\alpha)}\right)\\
 & = \frac{\Gamma(1+\lambda)}{\Gamma(1+\lambda + n(1-\alpha))} t^{\lambda + n(1-\alpha)}.
\end{align}
Thus, the relation is true for all $n\geq 1$.  \hskip 5cm $\blacksquare$
\end{proof}

\vskip 1cm

 Returning to $e_1(x,t)$,  we have by using Eq. \eqref{vim9}
\begin{align}\label{vim17}
\nonumber e_1(x,t) &=  c_1(x,t) - c_0(x,t) \\
\nonumber & = (Af)(x) J_t\  {}^C_0\mathcal{D}_t^{\alpha}(1) +  J_t q(x,t)\\
          & = (Af)(x) \frac{t^{1-\alpha}}{\Gamma(2-\alpha)} + \sum_{k\geq 0} q_k(x) \frac{t^{k+1}}{(k+1)!}.
\end{align}
Thus, we have
\begin{align}\label{vim18}
\nonumber e_{n+1}(x,t) & =A^n \left(J_t\ {}^C_0\mathcal{D}_{t}^{\alpha}\right)^n e_1(x,t)\\
\nonumber e_{n+1}(x,t) & =A^n \left(J_t\ {}^C_0\mathcal{D}_{t}^{\alpha}\right)^n \left((Af)(x) \frac{t^{1-\alpha}}{\Gamma(2-\alpha)} + \sum_{k\geq 0} q_k(x) \frac{t^{k+1}}{(k+1)!}\right)\\
\nonumber & = \left(A^{n+1}f\right)(x) \frac{1}{\Gamma(2-\alpha)} \left(J_t\ {}^C_0\mathcal{D}_{t}^{\alpha}\right)^n   t^{1-\alpha}\\
\nonumber & + \sum_{k\geq 0} \left(A^n q_k\right)(x) \frac{1}{(k+1)!} \left(J_t\  {}^C_0\mathcal{D}_{t}^{\alpha}\right)^n t^{k+1} \\
\nonumber & = \left(A^{n+1}f\right)(x) \frac{1}{\Gamma(1+(n+1)(1-\alpha))}  t^{(n+1)(1-\alpha)}\\
          & + \sum_{k\geq 0} \left(A^n q_k\right)(x) \frac{1}{\Gamma(2+k+n(1-\alpha))} t^{(k+1)+n(1-\alpha)}.
\end{align}
Hence we have the following lemma.
\begin{lemma}\label{lma2}
\begin{align}\label{vim19}
\nonumber e_{n+1}(x,t) & = \left(A^{n+1}f\right)(x) \frac{t^{(n+1)(1-\alpha)}}{\Gamma(1 + (n+1)(1-\alpha)} \\
          & + \sum_{k\geq 0} \left(A^n q_k\right)(x)  \frac{t^{(k+1)+n(1-\alpha)}}{\Gamma(2+k+n(1-\alpha))},
\end{align}
for all $n \geq 0$.
\end{lemma}
Thus, we can write by using Eq. \eqref{vim9}
\begin{equation}\label{vim20}
c_{n+1}(x,t) = f(x) + \sum_{j=0}^n e_{j+1}(x,t).
\end{equation}
\textbf{Remarks:}
\begin{enumerate}
\item If $q(x,t)$ depends only on the space variable $x$, then $q_k(x)=0$ for all $k \geq 1$ and $q(x,t) = q_0(x)$.
\item If $q(x,t)$ is of the form $q(x,t) = g(x)h(t)$, then take $q_k(x) = h^{(k)}(0) g(x)$, for all $k\geq 0$,
       where $h(t)$ is analytic at $t=0$.
\end{enumerate}

\section{A Case Study}\label{case}

\begin{figure}[t]
 \includegraphics[width=16cm]{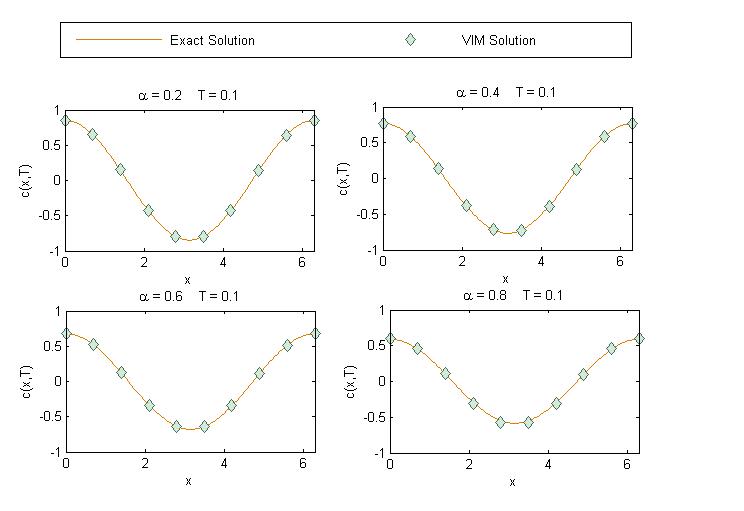}
  \caption{Plots of the exact solution $c(x,t)$, Eq. \eqref{case6}, and the truncated
    VIM solution $c_7(x,t)$, Eq. \eqref{case5},  at $t=0.1$ and $0 \leq x \leq 2\pi$.
    The plots corresponds $\alpha=$ $0.2, 0.4, 0.6$ and  $0.8$, as indicated.}
  \label{fig1}
  \end{figure}

\begin{figure}[t]
 \includegraphics[width=16cm]{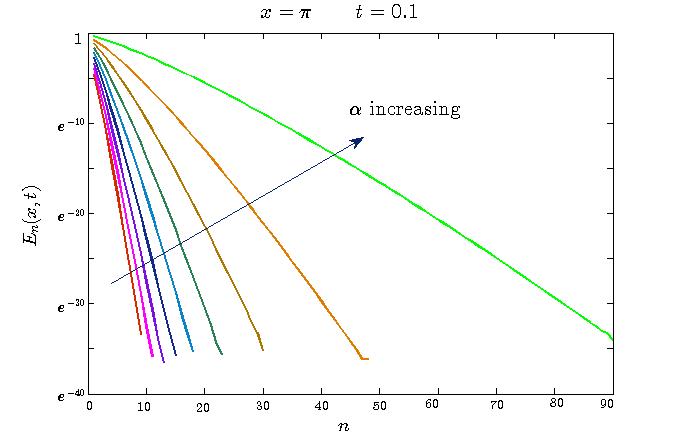}
  \caption{Plot of the relative error $E_n(x, t)$ at $x = \pi$, at $t=0.1$ against the number
of terms n. Vertical axis is scaled as Natural Logarithm, that is, $\ln e^m$, where $m \in \{-40, -35, \cdots , 0\}$.
}\label{fig2}
\end{figure}

\begin{figure}[b]
  \begin{minipage}[b]{0.5\linewidth}
  \includegraphics[width=16cm]{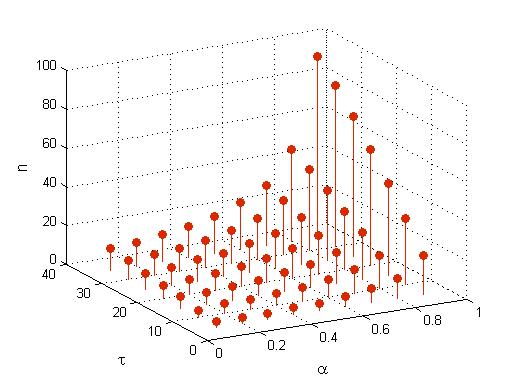}
  \end{minipage}
  \caption{ Stem plot for the Table \ref{Ta:1}. Plot depicts an increase in the values of $n$ as
  $\alpha$ increases and tolerance level becomes smaller.}\label{fig3}
\end{figure}

\begin{figure}[b]
  \begin{minipage}[b]{0.5\linewidth}
  \hspace{-1cm}
  \includegraphics[width=16cm]{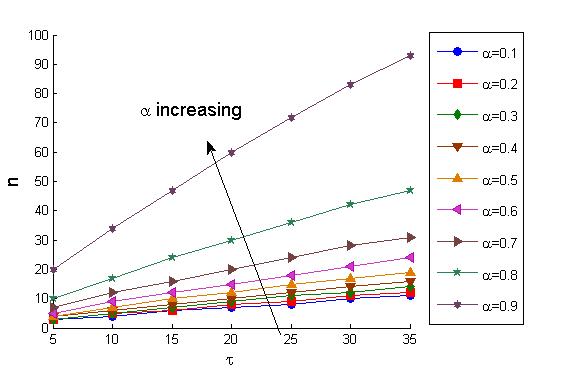}
  \end{minipage}
  \caption{Plots of $n$ against $\tau$, for specific values of $\alpha$.}\label{fig4}
\end{figure}

\begin{table}[t]
    \centering
    \begin{tabular}{*{10}{c}}\toprule
      \multirow{2}{*}[-0.5ex]{$\tau$} & \multicolumn{9}{c}{$\alpha$} \\ \cmidrule{2-10}
       & $0.1$ & $0.2$ & $0.3$ & $0.4$ & $0.5$ & $0.6$ &
$0.7$ & $0.8$ & $0.9$\\ \midrule
$5$ & $3$ & $3$ & $3$ &$4$ & $4$ & $5$ & $7$ & $10$ & $20$ \\
$10$ & $4$ & $5$ & $5$ & $6$ & $7$ & $9$ & $12$ & $17$ & $34$ \\
$15$ & $6$ & $6$ & $7$ & $8$ & $10$ & $12$ & $16$ & $24$ & $47$ \\
$20$ & $7$ & $8$ & $9$ & $10$ & $12$ & $15$ & $20$ & $30$ & $60$ \\
$25$ & $8$ & $9$ & $11$ & $12$ & $15$ & $18$ & $24$ & $36$ & $72$ \\
$30$ & $10$ & $11$ & $12$ & $14$ & $17$ & $21$ & $28$ & $42$ & $83$ \\
$35$ & $11$ & $12$ & $14$ & $16$ & $19$ & $24$ & $31$ & $47$ & $93$ \\
\bottomrule
    \end{tabular}
    \vskip 0.4cm
    \caption{For different values of $\alpha$ and different tolerance levels $\ln e^{-\tau}$, the $n$th
    approximate solution has to be used to achieve the required accuracy.}\label{Ta:1}
\end{table}

\subsection{The sinusoidal uploading}

Consider the following fractional differential equation

\begin{equation}\label{case1}
\frac{\partial{c(x,t)}}{\partial{t}}=
{}^C_0\mathcal{D}_t^{\alpha}\left[\frac{\partial^2c(x,t)}{\partial
x^2}\right]+t \sin x,\quad x
> 0, t > 0,
\end{equation}
with the initial condition $c(x,0)=\cos x$. On comparing with Eq. \eqref{intro1},
we find that the linear operator is $A=\frac{\partial^2}{\partial x^2}$, the initial data is $f(x)=\cos x$,
and the source term is $q(x,t)=t \sin x$, which on comparing with Eq. \eqref{vim11} yields
$q_0(x)=0$, $q_1(x)=\sin x$ and $q_k(x)=0$ for $k\geq 2$.

 We have following,
\begin{equation}\label{case2}
(A^{n+1}f)(x)=(-1)^{n+1}\cos x,
\end{equation}
and
\begin{equation}\label{case3}
(A^{n}q_1)(x)=(-1)^{n}\sin x.
\end{equation}
Substituting Eqs. \eqref{case2}-\eqref{case3} in Eq. \eqref{vim19}, we obtain
\begin{align}\label{case4}
\nonumber e_{n+1}(x,t) & = (-1)^{n+1}\cos x \frac{t^{(n+1)(1-\alpha)}}{\Gamma[1 + (n+1)(1-\alpha)]} \\ & + (-1)^{n}\sin x \frac{t^{2+n(1-\alpha)}}{\Gamma [3+n(1-\alpha)]}.
\end{align}
Substituting Eq. \eqref{case4} in Eq. \eqref{vim20}, we obtain
\begin{align}\label{case5}
\nonumber c_{n+1}(x,t) & = \cos x \left[1 + \sum_{j=0}^n (-1)^{j+1} \frac{t^{(j+1)(1-\alpha)}}{\Gamma[1 + (j+1)(1-\alpha)]}\right] \\
          & + t^2 \sin x \left[\sum_{j=0}^n (-1)^{j} \frac{t^{j(1-\alpha)}}{\Gamma[j(1-\alpha)+3]}\right].
\end{align}
Taking the limit $n\rightarrow\infty$, we obtain the exact solution,
\begin{equation}\label{case6}
c(x,t)=E_{1-\alpha}[-t^{1-\alpha}] \cos x + t^2 E_{1-\alpha, 3}[-t^{1-\alpha}] \sin x.
\end{equation}
\textbf{Remark:} Taking $\alpha =0$,  Eq. \eqref{case1} reduces to conventional partial
differential equation,
\begin{equation}\label{case7}
\frac{\partial{c(x,t)}}{\partial{t}}=
\frac{\partial^2c(x,t)}{\partial
x^2}+t \sin x,\quad x
> 0, t > 0.
\end{equation}
The solution of Eq. \eqref{case7} can be obtained by putting $\alpha =0$ in Eq. \eqref{case6},
that gives,
\begin{equation}\label{case8}
c(x,t)=E_{1}[-t^{1}] \cos x + t^2 E_{1, 3}[-t^{1}] \sin x,
\end{equation}
which agrees with the exact solution
\begin{equation}\label{case9}
c(x,t) = \exp(-t)\cos x + [\exp(-t)+t-1] \sin x.
\end{equation}

 Figure \ref{fig1} shows the plots of the exact solution \eqref{case6} and the VIM
approximate solution \eqref{case5} $c_7(x,t)$, i.e the truncated sum with $n=6$.
The graph shows close agreement between the exact and the VIM solutions.
Later, we will present error analysis, that is, error arises
when using truncated series as an approximate solution to the exact solution.

\subsection{Error Analysis}

Our next goal is to investigate the convergence of the approximate solutions obtained by the
VIM, Eq. \eqref{case5}. For this purpose, we define the relative  error as follows,
 \begin{equation}\label{exp15}
     E_n(x,t) = \frac{|c(x,t) - c_n(x,t)|}{|c(x,t)|}.
 \end{equation}

Figure \ref{fig2} shows the plots of relative error   at the point $(x,t)=(\pi, 0.1)$
against the numer of term $n$ in the truncated VIM solution, for different cases of
$\alpha =0.2, 0.4, 0.6, 0.8$. The vertical axis is scaled as Natural Logarithm.
The relative errors decay exponentially fast with $n$, but with convergence rates
 (slope of the plots in Figure 3) that decrease as $\alpha$ approaches 1.
Thus, a higher order approximate solution is required to achieve a given level of
accuracy as $\alpha$ increases.

Table \ref{Ta:1} summaries the results. It show the number of terms $n$ needed
for a given for a given accuracy, definded as $e^{-\tau}$, and for different
$\alpha$, and . We have depicted the information from Table \ref{Ta:1} as a stem
plot in Fig. \ref{fig3}, which shows
the trends in the number of terms $n$ against $\alpha$ and $\tau$.
$n$ increases for all $\alpha$ and $\tau$, but especially sharply as $\alpha$
approaches 1 and the tolerance becomes very small.
Moreover, $n$  increases almost linearly with
respect to tolerance level $\tau$ for a fixed value of $\alpha$, as shown in Fig. \ref{fig4}.

\section{Conclusions:}\label{conc}

 We have presented solutions of the time fractional diffusion equation with source term
\begin{equation*}
\frac{\partial}{\partial
t}c(x,t)={}^C_0\mathcal{D}_t^{\alpha}[(Ac)(x,t)]+q(x,t).
\end{equation*}
The solutions are found by using variational iteration method and are presented in the form
that avoids the repetition of calculations. The general form of the solutions, obtained by VIM,
is expressed in such a way so that it can be implemented on the computer with no difficulty.
Validation of the numerical procedure is done for a problem  whose exact solution is known.
Results obtained by VIM are in agreement with the exact solution.  It is shown that only
few successive approximations lead to a very good estimate of the exact solution. The
truncation errors decay exponentially fast as $n$ increases. VIM proves to be very efficient
and fast in finding the solutions of fractional differential equations.

%\begin{acknowledgements}
\subsection*{\large Acknowledgements}
The authors would like to acknowledge the support provided by King Abdulaziz City for Science and Technology (KACST) through the Science Technology Unit at King Fahd University of Petroleum and Minerals (KFUPM) for funding this work through project No.  11-OIL1663-04. as part of the National Science, Technology and Innovation Plan (NSTIP).

%National Science, Technology and Innovation Plan
%\end{acknowledgements}

%% BibTeX users please use one of

%\bibliographystyle{elsarticle-harv}
%\bibliographystyle{elsarticle-num-names}
%\bibliographystyle{elsarticle-num}
%\bibliographystyle{model1a-num-names}

%\bibliographystyle{model2-names}

%\bibliography{AMM_Nadeem}   % name your BibTeX data base

\end{document}